\documentclass[a4paper,11pt]{article}

\usepackage{authblk}
\usepackage{pst-all}
\usepackage{amsmath,amsthm,amssymb,mathrsfs,latexsym,amsfonts}
\usepackage{graphicx,psfrag,epsfig}
\usepackage[english]{babel}
\usepackage[latin1]{inputenc}
\usepackage{geometry}
{\begin{enumerate} \item[$\diamond$] \textsc{Proof :} }%
{\hfill $\square$\end{enumerate}}

\newtheorem{theoreme}{Theorem}[section]

\newtheorem{lemme}[theoreme]{Lemma}
\newtheorem{prop}[theoreme]{Proposition}

\newtheorem{cor}[theoreme]{Corollary}
\newtheorem{definition}[theoreme]{Definition\rm}

\newtheorem*{propositionA}{Proposition A}
\newtheorem*{propositionB}{Proposition B}

\begin{document}

\bibliographystyle{amsalpha}
\def\MP{\,{<\hspace{-.5em}\cdot}\,}
\def\SP{\,{>\hspace{-.3em}\cdot}\,}
\def\PM{\,{\cdot\hspace{-.3em}<}\,}
\def\PS{\,{\cdot\hspace{-.3em}>}\,}
\def\EP{\,{=\hspace{-.2em}\cdot}\,}
\def\PP{\,{+\hspace{-.1em}\cdot}\,}
\def\PE{\,{\cdot\hspace{-.2em}=}\,}
\def\N{\mathbb N}
\def\C{\mathbb C}
\def\Q{\mathbb Q}
\def\R{\mathbb R}
\def\T{\mathbb T}
\def\A{\mathbb A}
\def\Z{\mathbb Z}
\def\demi{\frac{1}{2}}
\def\d{\Delta}
\def\Max{\rm}
\def\norm{{\rm norm \,}}
\def\conorm{{\rm conorm \,}}
\def\a{\alpha}
\def\demi{\frac{1}{2}}
\def\d{\Delta}
\def\t{\theta}
\def\graph{{\rm graph \,}}
\def\Im{{\rm Im \,}}

\begin{titlepage}
\author{Lara Sabbagh \thanks{This work was partially supported by the EPSRC grant EP/J003948/1.}}
\title{\LARGE{\textbf{An inclination lemma for normally hyperbolic manifolds with an application to diffusion}}}
\affil{Mathematics Institute, University of Warwick, email: l.el-sabbagh@warwick.ac.uk}
\end{titlepage}

\maketitle

\begin{abstract}Let ($M$, $\Omega$) be a smooth symplectic manifold and $f:M\rightarrow M$ be
a symplectic diffeomorphism of class $C^l$ ($l\geq 3$). Let $N$ be a compact submanifold of $M$ which is boundaryless and normally hyperbolic for
$f$. We suppose that $N$ is controllable and that its stable and unstable bundles are trivial. We consider a
$C^1$-submanifold $\d$ of $M$ whose dimension is equal to the
dimension of a fiber of the unstable bundle of $T_NM$. We suppose
that $\d$ transversely intersects the stable manifold of $N$. Then,
we prove that for all $\varepsilon>0$, and for $n$ $\in$ $\N$ large
enough, there exists $x_n$ $\in$ $N$ such that $f^n(\d)$ is
$\varepsilon$-close, in the $C^1$ topology, to the strongly unstable
manifold of $x_n$.  

As an application of this $\lambda$-lemma, we prove the existence of shadowing orbits for a
finite family of invariant minimal sets (for which we do not assume
any regularity) contained in a normally hyperbolic manifold and having heteroclinic connections. As a particular case,
we recover classical results on the existence of diffusion orbits
(Arnold's example).

\end{abstract}

\bigskip



\section{Introduction}
In his famous note \cite{Arn64}, Arnold gave the first example of a
three-degree-of-freedom system where diffusion orbits shadowing
whiskered tori were constructed. More precisely, the system admits
orbits for which the action undergoes a drift of length independent
of the size of the perturbation. Arnold's example was chosen so that
the Lagrangian invariant tori in the unperturbed system break down 
under the perturbation and  give rise to  partially hyperbolic tori in the 
perturbed system. 

The diffusion mechanism is then based on the existence of
a transition chain, that is, a family of invariant minimal tori with
heteroclinic connections. One gets the orbits shadowing the extremal tori of 
this chain by an ``obstruction argument'' satisfied by each torus of the chain. 
This obstruction argument was first 
proved in the paper \cite{Mar96} as a corollary of a partially hyperbolic 
$\lambda$-lemma. The proof was then improved in \cite{FM}  (see also \cite{C00}).

In the present  paper, we prove a $\lambda$-lemma (also called inclination lemma) for normally hyperbolic
invariant manifolds, which turns out to be a new tool for proving the
obstruction argument as well as several generalizations. This
$\lambda$-lemma deals with normally hyperbolic manifolds instead of partially hypebolic tori.
This is not a genuine restriction since one can in general embed partially hyperbolic tori into
their central manifolds which, as a rule, are normally hyperbolic. In that respect, this paper
generalizes the results of \cite{Mar96}, \cite{C00} and
\cite{FM}, and enables us to significatively simplify the previous proofs. Moreover, our $\lambda$-lemma can be applied to more general systems than that of
Arnold (\cite{DLS}, \cite{DH}, \cite{GR}, \cite{GR09},...) and can relate to the variational methods (which is another approach to diffusion problems) where significant contributions were given by Bernard, Bessi, Cheng, Kaloshin and many others.

We first state and prove a $\lambda$-lemma for normally
hyperbolic invariant manifolds. Given a normally hyperbolic
invariant manifold $N$  for a diffeomorphism $f$, we consider a
submanifold that transversely intersects the stable manifold of $N$
and whose dimension is equal to the dimension of a fiber of the
unstable bundle. We prove that under iteration by $f$, this
submanifold is as close  as desired (in the $C^1$ topology) to a
suitable unstable leaf. The $\lambda$-lemma will enable us to
prove the existence of drifting orbits along a chain of invariant
minimal sets contained in a normally
hyperbolic manifold, without any assumption on the nature of the invariant sets
(in particular, they do not need to be submanifolds). 
As an easy particular case, we recover Arnold's example.
In addition, the $\lambda$-lemma applies immediately to the
examples  of Delshams, De La Llave and Seara (see \cite{DLS} and 
the references therein) and yields the diffusion orbits.


In this paper, we will limit ourselves to the symplectic case and we will assume that our normally hyperbolic manifold
has  trivial stable and unstable bundles 
(this will in particular give us  easy regularity conditions for the lamination of the invariant
manifolds). This will be no restriction to us since all the applications that we have in mind will 
fall into this category (diffusion orbits, Easton's windows,...). Moreover, we will adopt a very 
basic point of view and depict the geometry of the iterates of our transverse manifolds instead
of using a more synthetic method (fixed point theorem for instance). In particular, this will 
enable us to directly use our various computations for the construction of windows and for
estimating the transition times in a subsequent work. As a counterpart, we will have to use 
the existence of ``controlled'' straightening neighborhoods for our manifold, which requires 
the previous (maybe unnecessary) assumptions.

The normally hyperbolic invariant manifolds we consider will be
compact for technical simplicity but the non-compactness could
easily be replaced with uniform lower bounds for the first and
second derivatives of our diffeomorphisms and the constants of
hyperbolicity (see (\ref{VNH}) below).
Finally, let us point out that eventhough we  prove the $\lambda$-lemma for discrete systems, 
as usual analogous results hold for the continuous time Hamiltonian systems.

%

\vskip2mm

\emph{Acknowledgments.}
I would like to thank Jean-Pierre Marco for having suggested these questions to me and for having generously shared his ideas with me.

\section{A reminder on normally hyperbolic invariant manifolds and convention}
We begin with a reminder on normally hyperbolic manifolds in a general context and then specialize to the symplectic case where we can use a ``controlled'' straightening neighborhood in which it is easy to depict the geometry of the invariant foliations induced by normal hyperbolicity.

\subsection{General definitions}\label{GD}

Let $M$ be a smooth $n$-dimensional manifold ($n\geq3$) and
$f:M\rightarrow M$ be a $C^l$-diffeomorphism ($l\geq 1$) which
leaves a smooth boundaryless compact submanifold $N$ of $M$
invariant. Given a Riemannian metric $\parallel.\parallel$ on $M$
and a subbundle $E$ of $T_NM$ invariant under $Df$, we set:
$$\norm(Df_{|_E})=\sup\{\|Df(a)_{|_{E_a}}\|; a \in N\},\quad
\conorm(Df_{|_E})=(\norm(Df^{-1}_{|_E}))^{-1}.$$


\begin{definition} Let $q \leq l$ ($q$ $\in$ $\N^*$). The manifold $N$ is $q$-normally
hyperbolic for $f$ if the tangent bundle of $M$ restricted to $N$
splits into three continuous subbundles $T_NM=TN\oplus E^s \oplus
E^u$ invariant under $Df$, such that
\begin{equation}\label{VNH}\norm(Df_{|_{E^s}})<(\conorm(Df_{|_{TN}}))^q\leq1\leq
(\norm(Df_{|_{TN}}))^q<\conorm(Df_{|_{E^u}}).\end{equation}
\end{definition}

This says that the behavior of $f$ normal to $N$ dominates the
tangent behavior of $f^q$ and is hyperbolic.

Now we state the local stable/unstable manifolds theorem. We do not
mean to give the most general possible results, we rather limit
ourselves to those which are strictly necessary for our purposes.
For a more elaborate study on invariant manifolds, we refer to
\cite{HPS}, \cite{Cha} and \cite{BB}.

\bigskip

\noindent\textbf{Theorem \cite{HPS}}. \textit{Let $f$, $M$ and $N$
be as above. We suppose that $N$ is $q$-normally hyperbolic for $f$.
Then if $d$ is the distance associated with the Riemannian metric on
$M$, the following properties hold true:}

\vspace{0.2cm}
\noindent\textbf{ 1. Existence, characterization and smoothness.} \textit{There
exists a neighborhood $\mathcal{O}$ of $N$ in $M$ such that the
sets}:
\[
W^s_{loc}(N)=\Big\{y\in \mathcal{O}\,; f^n(y)\in \mathcal{O},\,\forall n \in \N\Big\}
\:\:\:{{\textrm and}}\:\:\:
W^u_{loc}(N)=\Big\{y\in\mathcal{O}\,;\,f^{-n}(y)\in \mathcal{O},\, \forall n \in \N\Big\}
\]
\textit{are $C^q$-manifolds that satisfy}
\begin{itemize}
\item $\forall \,y\, \in\,W^s_{loc}(N), \forall \,\rho\, \in\, \left]\, \norm(Df_{|_{E^s}});\, \conorm(Df_{|_{TN}})\,\right[,\lim\limits_{n\rightarrow\infty}\rho^{-n}d(f^n(y),N)=0$,
\item$\forall \,y\, \in\,W^u_{loc}(N), \forall \,\rho\, \in\,\left]\,\norm(Df_{|_{TN}});\,\conorm(Df_{|_{E^u}})\,\right[,
\lim\limits_{n\rightarrow \infty}\rho^{n}d(f^{-n}(y),N)=0$.
\end{itemize}
\textit{Moreover, $W^u_{loc}(N)$ and $W^s_{loc}(N)$ are tangent to $TN\oplus
E^u$ and $TN\oplus E^s$ respectively at each point of $N$}.
\vspace{0.2cm}

\noindent \textbf{2. Lamination.} \textit{There exist two $f$-invariant
laminations of $W^u_{loc}(N)$ and $W^s_{loc}(N)$, the leaves of
which are unstable and stable leaves $W^{uu}_{loc}(x)$ and $W^{ss}_{loc}(x)$
associated with the points of $N$, defined as follows:
$$W^{ss}_{loc}(x)=\left\{y\in \mathcal{O}\,; \lim\limits_{n\rightarrow \infty}d\left(f^{n}(y),f^{n}(x)\right)=0\right\} \text{ and }$$ $$W^{uu}_{loc}(x)=\left\{y\in \mathcal{O}\,; \lim\limits_{n\rightarrow \infty}d\left(f^{-n}(y),f^{-n}(x)\right)=0\right\}.$$ These leaves are $C^q$ and tangent
to the fibers $E^u_x$ and $E^s_x$ at each point $x$ of $N$}.



\bigskip

Note that one gets the global stable (resp. unstable) manifolds by
taking the union of the inverse (resp. direct) images of the local
ones as follows:
$$W^s(N)=\bigcup\limits_{n\in \N} f^{-n}\left(W^s_{loc}\left(N\right)\right)\,\,\text{ and }\,\,W^u(N)=\bigcup\limits_{n\in \N} f^{n}\left(W^u_{loc}\left(N\right)\right). $$ The same holds for the leaves:$$W^{ss}(x)=\bigcup\limits_{n\in \N} f^{-n}\left(W^{ss}_{loc}\left(f^{n}(x)\right)\right)\,\,\text{ and }\,\,W^{uu}(x)=\bigcup\limits_{n\in \N} f^{n}\left(W^{uu}_{loc}\left(f^{-n}(x)\right)\right). $$

These are immersed $C^q$-submanifolds of $M$. In the rest of the paper, we will drop the subscript $loc$
from the notation. The local and
the global invariant manifolds will be denoted by $W^{s,u}(N)$ since the context will always be clear. The same holds for the global and local leaves.

\begin{definition}Let $N$ be a $q$-normally hyperbolic manifold for $f$ ($q\leq l$). We say that $N$ is controllable if the following inequalities hold true
\begin{equation} \norm(Df_{|_{E^s}}).\norm(Df_{|_{TN}})<1 \,\,\,\text{ and } \,\,\,
\conorm(Df_{|_{TN}}).\conorm(Df_{|_{E^u}})>1.
\end{equation}
\end{definition}

\bigskip

We set $n_s:=$dim($E^s$), $n_u:=$dim($E^u$) and $n_0:=$dim($N$), so
that $n_0+n_s+n_u=n$.\\

\subsection{Symplectic Geometry and normal hyperbolicity}
Under symplecticity assumptions, the stable and unstable
leaves are regular with respect to the points in $N$. More precisely, we have the following proposition which will enable us in the next section to introduce a \emph{straightening} coordinate system in the vicinity of normally hyperbolic manifolds.



\begin{propositionA} \emph{\textbf{[Marco].}} Let ($M$, $\Omega$) be a smooth symplectic manifold and let $f$ be a $C^l$ symplectic diffeomorphism of $M$ ($l\geq 2$). We suppose that $N$ is a controllable $q$-normally hyperbolic manifold for $f$ ($q\leq l$). Then
\begin{itemize}
\item [-] $N$ is symplectic,
\item [-] $W^u(N)$ and $W^s(N)$ are coisotropic,
\item [-] $n_s=n_u$,
\item [-] for all $x$ $\in$ $N$, $W^{uu}(x)$ and $W^{ss}(x)$ are isotropic
and they coincide with the leaves of the characteristic foliations
of $W^u(N)$ and $W^s(N)$.
\end{itemize}
\end{propositionA}

The proof of this proposition can be found in \cite{Ma}. Since the
leaves of the characteristic foliations coincide with the leaves
$W^{uu}(x)$ and $W^{ss}(x)$, the latter are $C^{q-1}$ with respect
to $x$. We get then the regularity we need for Proposition B below.

\subsection{Straightening neighborhood and convention}
Under the assumptions of Proposition A, one can find in the
vicinity of a normally hyperbolic manifold a
neighborhood in which the invariant manifolds and the leaves are
straightened, making it easier to depict the behavior of $f$. More
precisely, we have the following proposition.

\begin{propositionB}\emph{\textbf{[Tubular neighborhood and straightening].}}\label{tub}
Let $M$, $N$ and $f$ be as in Proposition~A with $l\geq 3$. Let $p:=n_s=n_u$. We suppose that $N$ is $3$-normally hyperbolic for $f$ and that
its stable and unstable bundles are trivial. Then, there exist a neighborhood $U$
of $N$ in $M$ and a $C^2$-diffeomorphism $\varphi : U \longrightarrow
V:=N\times B^p\times B^p$, where $B^p$ is an open ball centered at
$0$ in $\R^{p}$, such that for all $x$ $\in$ $N$:

\begin{enumerate}
\item $\varphi(x)=(x,0,0)$,
\item $\widetilde{W}^s(N):=\varphi(W^s(N)\cap U)=\{(x,s,u) \in V\,;\, u=0\}$,
\item $\widetilde{W}^u(N):=\varphi(W^u(N)\cap U)=\{(x,s,u) \in V\,;\, s=0\}$,
\item $\widetilde{W}^{ss}(x):=\varphi(W^{ss}(x)\cap U)=\{(x,s,0)\,;\, s \in B^p\}$,
\item $\widetilde{W}^{uu}(x):=\varphi(W^{uu}(x)\cap U)=\{(x,0,u)\,;\, u
\in B^p\}$.
\end{enumerate}

\end{propositionB}

The proof is straightforward once Proposition A is known. We will
not prove Proposition B, we will content ourselves with the
following few remarks. Near $N$, one can always find a tubular
neighborhood. The straightening of the invariant manifolds is an
immediate consequence of the graph property. We refer to \cite{LMS}
and \cite{HPS} for details. When $f$ is symplectic, the strongly
stable/unstable leaves are straightened the same way.

\begin{figure}[h]
\centering
\def\JPicScale{0.6}
\input{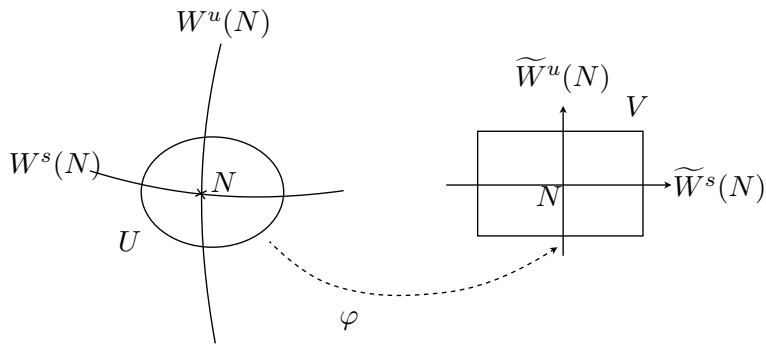}
\caption{The straightening neighborhood}
\end{figure}

\bigskip

\noindent{\bf Convention.} The first $B^p$ and the second $B^p$ in $N\times B^p\times B^p$ do
not play the same role since the first one is the stable direction while the second
one is the unstable direction. In order to distinguish them from one another
when we want to use them separately, we will add $u$
and $s$ in the notation as follows

\begin{equation}\label{p}
N\times B_s^p\times B_u^p.
\end{equation}
\noindent We use the same convention for $N\times \R_s^p\times \R_u^p$. 

In the rest of the paper, we will identify $N$ with $\varphi(N)=N\times\{0\}\times\{0\}$ for notational symplicity. This will not lead to confusion since the context will always be clear enough.

 Let $d$ be the distance associated with the Riemannian metric on $M$. We will equip the neighborhood $V$ defined in Proposition~B with the distance given by the sup of $d_{\mid_N}$ and the Euclidian distance on $\R^{2p}$. It is equivalent to the image under $\varphi$
of $d$ since $V$ is relatively compact. 

We will use the usual operator norms for the linear applications
defined on Banach spaces that we will deal with throughout the
paper. We will equip the product spaces with the sup norm and the
subsets with the induced norm. For notational simplicity, we will
denote all our norms by the same symbol $\|.\|$; the context will
always be clear enough to avoid ambiguities.

To prove our results, we will use compositions of linear applications defined on the tangent spaces of some suitable manifolds. They will be normed algebras for the induced norm.

\section{A $\lambda$-lemma for normally hyperbolic manifolds}\label{jfk}

In this section, we prove a  $\lambda$-lemma for normally hyperbolic manifolds. From now on, we suppose that $f$, $M$, $N$ and $\d$ are as follows:
\begin{itemize}
\item ($M$, $\Omega$) is a smooth symplectic Riemannian manifold,
\item $f:M\longrightarrow M$ is
a symplectic diffeomorphism of class $C^l$ ($l\geq 3$),
\item $N$ is a smooth submanifold of $M$, compact and boundaryless,
\item $N$ is a controllable 3-normally hyperbolic manifold for
$f$,
\item $n_s=n_u=p$,
\item  $N$ has trivial stable and unstable bundles,
\item $\d$ is a $C^1$-submanifold of $M$ of dimension $p$ which
transversely intersects $W^s(N)$ at some point $a$.  
\end{itemize}

We will state two versions of the $\lambda$-lemma. In Section~\ref{not}, we will use the straightening neighborhood given in Section~\ref{tub} to set out a simplified version of the $\lambda$-lemma (Theorem~1) and to properly define the notion of $C^1$-convergence. Then, in Section~\ref{inKbis}, we state the $\lambda$-lemma in a more general context (Theorem~2). We devote Sections~\ref{inV} and \ref{inK} to the proofs of these theorems.



\subsection{Theorem  1: in the straightening neighborhood}\label{not}

In this section, we state the $\lambda$-lemma in the straightening neighborhood. Let us start with fixing the notation. We keep the notation of Proposition B. We will restrict our
diffeomorphism $\varphi$ to the open set $\mathcal{U}:=U\cap
f^{-1}(U)$, so that $F=\varphi \circ f \circ \varphi^{-1}$ is well
defined on $\mathcal{V}:=\varphi(\mathcal{U})\subset V$ with values
in $V$. A point in $V$ will be written as a triple $(x,s,u)$ and
$F$ as $(F_x, F_s, F_u)$ according to the splitting $V=N\times
B_s^p\times B_u^p$. Up to iterating $\d$ if necessary (and resetting the counters), we can suppose that $a$ $\in$ $\mathcal{U}$ without loss of generality, since we are interested in the behavior of $\d$ after a large number of iterations.


We introduce the projection $\Pi_N:\widetilde{W}^s(N)\longrightarrow
N$ that sends each $(x,s,0)$ to $(x,0,0)$. Let $P:=\varphi(a)=(x,s,0)$ be the
intersection point of $\varphi(\d \cap \mathcal{U})$ and $\widetilde{W}^s(N)$.
We set $P_0:=\Pi_N(P)$. For $n\geq1$, we denote by $P^n=F^n(P)$, and
$P^n_0:=\Pi_N(P^n)=F_{\mid_N}^n(P_0)$ which is the point in $N$ such that $P^n \in
\widetilde{W}^{ss}(P^n_0)$ (see Figure \ref{redressement}). We
denote by $\widetilde{\d}$ the connected component of $\varphi(\d
\cap \mathcal{U})$ in $\mathcal{V}$ containing $P$. For all $n \in
\N $, we denote by 
$\widetilde{\d}^{n+1}$ the connected component of $F(\widetilde{\d}^n)\cap\mathcal{V}$ containing $P^n$ (where
$\widetilde{\d}^0=\widetilde{\d}$).

\begin{definition}\emph{\textbf{[The graph property].}}\label{graph}
Let $\Lambda$ be a $C^1$-submanifold of $N\times\R_s^p\times\R_u^p$.
Let $B$ be an open ball in $\R_u^p$. We say that $\Lambda$ has the
\emph{graph property over $B$}, or equivalently that $\Lambda$ is a \emph{graph
over $B$}, if there exists a $C^1$-map $\varpi: B\rightarrow N\times
\R_s^{p}$ such that $\Lambda=\{(\varpi(u),u);u\in B\}$.

\end{definition}

\bigskip

For $\delta$ small enough, we set $B_{\delta}:=\{u\in B_u^p\,;\, \|
u\|<\delta\}$ and $D_{\delta}:=\{(x,s,u) \in \mathcal{V}\,;\, u \in
B_{\delta}\}$.  For $n$ $\in$ $\N$, we introduce the constant map
$$\begin{array}[t]{cccc}\ell_n: & B_{\delta}& \longrightarrow &
N\times B_s^p\\ &u &\longmapsto& (P^n_0,0)
\end{array}$$so that clearly $\widetilde{W}^{uu}(P_0^n)\cap D_{\delta}$
is the graph of $\ell_n$, for all $n$ $\in$
$\N$.

\begin{figure}[h]
\centering
\def\JPicScale{0.8}
\input{graphproperty2.pst}
\caption{Graphs}
\end{figure}


The $\lambda$-lemma in $V$ takes the following form.

\bigskip

\noindent\textbf{Theorem 1}. \label{thmV} \textit{For all $n$ $\in$
$\N$, let $\widetilde{\d}^n$ and $\ell_n$ be as above. Then, there
exists $\delta>0$ such that for all $n$ $\in$ $\N$, there exists a
$C^1$-map $\xi_n : B_{\delta}\rightarrow N\times B_s^p$ such that
$\widehat{\d}^n:=\widetilde{\d}^n\cap D_{\delta}$ is the graph of
$\xi_n$. Moreover,}
$$\lim\limits_{n\rightarrow
\infty}d_{C^1}(\xi_n,\ell_n)=0,$$where $d_{C^1}(\xi_n,\ell_n)=\sup\limits_{u \in B_{\delta}}\big(d(\xi_n(u),\ell_n(u))+\|\xi'_{n}(u)-\ell'_n(u)\|\big)$.

\bigskip

\begin{figure}[h]
\centering
\def\JPicScale{0.8}
\input{V4.pst} \vspace{0.7cm} \caption{Straightening of $\widetilde{\d}$}
\label{redressement}
\end{figure}






We will need $4$ steps to prove Theorem 1 in Section \ref{inV}. We
will first show how, under iteration, arbitrary tangent vectors in
$T_{P_0}\widetilde{\d}$ are straightened. We will then use the
transversality of $\widetilde{\d}$ to $\widetilde{W}^s(N)$ to prove
that some suitable part of $\widetilde{\d}$ (close to $P$) is a graph over a ball in $\R_u^{p}$. In the
third step, we will show how this graph property is preserved under
iteration \textit{over the same domain in $\R_u^{p}$}. We will
finally prove that tangent vectors along these graphs are
straightened and a simple application of the Mean Value Theorem ends
the proof of
Theorem 1.\\

We end this section with the definition of a notion of
``closeness" for graphs which will be useful in the sequel.
\begin{definition}\label{epsilunclose} We keep the notation of
Theorem 1. Let $\varepsilon>0$ and $n$ $\in$ $\N$. We say that
$\widehat{\d}^n$ and $\widetilde{W}^{uu}(P_0^n)\cap D_{\delta}$ are $C^1$
$\varepsilon$-close if\, $d_{C^1}(\xi_n,\ell_n)<\varepsilon$.

\end{definition}

\subsection{Theorem 2: in an arbitrary compact subset of
$M$}\label{inKbis} In this section, we introduce a new notion of
graphs and convergence in the $C^1$ compact open topology (in a fixed relatively compact set in
$M$).

\begin{definition} Let \,$\mathcal{U}$ and $\overline{U}$ be two neighborhoods of $N$ in $M$
such that $\overline{U}\subset\mathcal{U}$. We suppose that there
exists a $C^2$-diffeomorphism $\varphi: \mathcal{U}\longrightarrow N\times \R_s^p\times
\R_u^p$. Let $m$ $\in$ $\N$ be fixed. We set
$\psi_{(m,\overline{U})}:=f^m\circ{\varphi^{-1}}_{|_{\varphi(\overline{U})}}$.
Let $Q_1$ be a $C^1$-submanifold of $M$ contained in $W^u(N)\cap
f^m(\overline{U})$ and $Q_2$ be a $C^1$-submanifold of $M$ contained in
$f^m(\overline{U})$. We say that $Q_2$ is a $(m,\overline{U})$-graph
over $Q_1$ if $\psi_{(m,\overline{U})}^{-1}(Q_2)$ is a graph over
$\Pi_3\left(\psi_{(m,\overline{U})}^{-1}(Q_1)\right)$ in the
sense of Definition \ref{graph}, where $\Pi_3$ denotes the
projection over the third variable.

 If $\psi_{(m,\overline{U})}^{-1}(Q_2)=\graph \xi=\graph(X,S)=\{(X(u),S(u),u);u
\in \Pi_3(\psi_1^{-1}(Q_1)) \}$, we define the following distance
$$\begin{array}{ll}d_{(C^1,m,\overline{U})}\big(Q_1,Q_2\big):=&
\sup\limits_{u \in
\Pi_3\left(\psi_{(m,\overline{U})}^{-1}(Q_1)\right)}d\left(\psi_{(m,\overline{U})}\left(X(u),S(u),u\right),
 \psi_{(m,\overline{U})}(X(0),0,u)\right)+\end{array}$$
 $$\sup\limits_{ \substack {u\in
 \Pi_3(\psi_{(m,\overline{U})}^{-1}(Q_1))\\ v_1 \in B_{\R_u^{p}}}}\left\|{D\psi_{(m,\overline{U})}}\left(X(u),S(u),u\right).
 \left(X'(u).v_1,S'(u).v_1,v_1\right)-{D\psi_{(m,\overline{U})}}\left(X(0),0,u\right).\left(0,v_1\right)\right\|$$
where $B_{\R_u^{p}}$ is the unit ball in $\R_u^{p}$.
\end{definition}

\begin{figure}
\centering
\def\JPicScale{0.8}
\input{defcompactK.pst}
\caption{The $(m,\overline{U})$-graph property}
\end{figure}

We now state the global version of the $\lambda$-lemma.\\




\noindent\textbf{Theorem 2}. \textbf{[$\lambda$-lemma]}.\label{lambdalemma}\label{K} \textit{Let $f$, $M$ and $N$ be as above. Let $\d$ be a $p$-dimensional
$C^1$-submanifold transversely intersecting $W^s(N)$ at some point
$a$, and let $\d^k=f^k(\d)$, for $k\geq 1$. Let $a_0$ be the point
in $N$ such that $a \in W^{ss}(a_0)$ and set $a_0^k:=f^k(a_0)$.}

\textit{Then, there exist two neighborhoods $\mathcal{U}$ and
$\overline{U}$ of $N$ in $M$ satisfying
$\overline{U}\subset\mathcal{U}$, and a $C^2$-diffeomorphism
$\varphi: \mathcal{U}\longrightarrow N\times \R_s^p\times
\R_u^p$ such that
$\forall m \in \N$, $\forall \varepsilon
>0, \exists k_0~\in~\N$; $\forall k\geq k_0$, there exists a $C^1$-submanifold $\overline{\d}^k$ in $f^k(\d)\cap f^m(\overline{U})$ such that $\overline{\d}^k$ is a
$(m,\overline{U})$-graph over $W^{uu}(a_0^k)\cap f^m(\overline{U})$.
Moreover,}
$$d_{(C^1,m,\overline{U})}\left(\overline{\d}^k,W^{uu}(a_0^k)\cap f^m(\overline{U})\right)<\varepsilon.$$


We devote Section \ref{inK} to the proof of Theorem 2. It will be a
direct consequence of the proof of Theorem 1.\\

\noindent\emph{\textbf{Comments.}} Theorem 2 actually states the
straightening property in any relatively compact set $K$ (with a
non-empty interior) in $M$ intersecting all the unstable leaves of the submanifold $N$. More precisely, let $K$ be such a set. The sequence
$(f^m(\overline{U})\cap W^u(N))_{m\in \N}$ is clearly an exhaustion of $W^u(N)$
by relatively compact sets. By definition of the unstable manifold,
there exists an integer $m_0$ such that $W^u(N)\cap K \subset
f^{m_0}(\overline{U})$. Then, one can easily prove that for all
$\varepsilon
>0$, there exists $k_0 \in \N$ such that for all $k\geq k_0$, there exists a submanifold $\underline{\d}^k$ in $f^k(\d)\cap
K$ such that $\underline{\d}^k$ is a $(m_0,\overline{U})$-graph over
$W^{uu}(a_0^k)\cap K$. Moreover,
$$d_{(C^1,m_0,\overline{U})}\left(\underline{\d}^k,W^{uu}(a_0^k)\cap K\right)<\varepsilon.$$

Note that the convergence given by the basic $\lambda$-lemma is
stronger than the Hausdorff one, for $\Delta$ and for its tangent
space as well.

\section{Proof of Theorem 1}\label{inV}
In this section, we prove Theorem~1.

\subsection{General assumptions for Theorem 1}
Here we
keep the notation of Proposition B and of Section \ref{not} and we limit ourselves to the behavior of $F$ in $\mathcal{V}$. Recall
that $\mathcal{V}\subset V=N\times B_s^p\times B_u^p$, where
$B_{s,u}^p$ is an open ball centered at $0$ in $\R^{p}$. Let
$B_{s,u}^p$ be of radius $\varsigma$.

Since $\widetilde{W}^{s,u}(N)$ are invariant under $F$,
then

\begin{equation}\label{Fu}\forall \ x \in N, \,\,\forall \ s \in B_s^p, \,\,\ F_u(x,s,0)=0,\end{equation}
\begin{equation}\label{Fs}\forall \ x \in N, \,\,\forall \ u \in B_u^p,\ \,\, F_s(x,0,u)=0.\end{equation}
In addition, $\forall \ x \in N$, $F(x,0,0)=(F_x(x,0,0),0,0)$. Since the
stable and unstable \emph{foliations} are invariant, then for all
$(x,s,u)$ $\in$ $\mathcal{V}$,
\begin{equation}\label{Fx}F_x(x,0,u)=F_x(x,s,0)=F_x(x,0,0).\end{equation}


\noindent Therefore, for $X=(x,0,0)$ $\in$ $N\times\{(0,0)\}$, the
derivative $DF(X)$ at $X$ can be represented as a diagonal matrix:
\begin{equation}\label{DFdiag}DF(X)=\left(\begin{array}{ccc}
\partial_xF_x(X)& 0 &  0\\
0 & \partial_sF_s(X) & 0\\
0& 0 & \partial_uF_u(X)
\end{array}\right ).\end{equation}

\noindent The manifold $N\times\{(0,0)\}$ being normally hyperbolic
for $F$, one can find a real number $\lambda$ $\in$ $]0;1[$ such that
$\forall$ $x$ $\in$ $N$,

$$\|\partial_sF_s(x,0,0)\|<\lambda, \\\ \|(\partial_uF_u)^{-1}(x,0,0)
\|<\lambda,\\\
\|\partial_sF_s(x,0,0)\|.\|(\partial_xF_x(x,0,0))^{-1}\|<\lambda$$
$$\text{
and
}\|\partial_xF_x(x,0,0)\|.\|(\partial_uF_u(x,0,0))^{-1}\|<\lambda.$$
\bigskip

\noindent Let $Y=(x,s,0)$ be in $\widetilde{W}^{s}(N)$. Using
Equations (\ref{Fu}) and (\ref{Fx}), one easily sees that $DF(Y)$
takes the following form:

\begin{equation}\label{DFS}DF(Y)=\left(\begin{array}{ccc}
\partial_xF_x(Y)& 0 &  \partial_uF_x(Y)\\
\partial_xF_s(Y) & \partial_sF_s(Y) & \partial_uF_s(Y)\\
0& 0 & \partial_uF_u(Y)
\end{array}\right ).\end{equation}

\bigskip


\noindent One has an analogous property for the points of $\widetilde{W}^u(N)$.

 We  need to shrink $\mathcal{V}$ in order to have some
estimates useful later on. Note first that $\mathcal{V}$ can be
chosen so that $\partial_xF_x(Z)$ and $\partial_uF_u(Z)$ are
invertible for all $Z$ $\in$ $\mathcal{V}$.\\

 Let $\overline{\lambda}$ be in $]\lambda;1[$. For
simplicity, we choose $\overline{\lambda}=\frac{1+\lambda}2$.
However, all the calculations in this proof can be adjusted so that
they are compatible with any value of $\overline{\lambda}$ $\in$
$]\lambda;1[$. Recall that $B_{s,u}^p$ is of radius $\varsigma$.
\begin{prop}\label{vois}
For $\varsigma$ small enough, there exist real positive constants
$C_1$ and $C_2$ such that for all $Z=(x,s,u)$ $\in$ $\mathcal{V}$,
the following inequalities hold true

\begin{enumerate}

\item $\|s\|<\frac{5-5\lambda}{2C_2(11+\lambda)}$,

\item $\|DF(Z)\|\leq C_1$ and $\|D^2F(Z)\|\leq C_2$,

\item $\|\partial_sF_s(Z)\|<\overline{\lambda}$ and $\|[\partial_uF_u(Z)]^{-1}\|<\overline{\lambda}$,

\item $\|\partial_xF_x(Z)\|.\|[\partial_uF_u(Z)]^{-1}\|<\overline{\lambda}$,

\item
$\max\big(\|\partial_sF_x(Z)\|,\|\partial_xF_s(Z)\|\big)<\frac{5-5\lambda}{2(11+\lambda)}$.

\end{enumerate}

\end{prop}

\begin{proof}

 The proof is immediate because $F$ is at least $C^2$ and $\mathcal{V}$ is relatively compact. Note that the last item is immediate thanks to the form of $DF$ in Equation~(\ref{DFdiag}).
\end{proof}


\subsection{Linear straightening of $T_{P^m}\widetilde{\d}^m$}\label{step1}
The following proposition states the straightening of the tangent
space of $\widetilde{\d}$ at its base point, under iteration by $F$.
\begin{prop}\label{etape1} For all $m$ $\in$ $\N$, the tangent space $T_{P^m}\widetilde{\d}^m$ is the graph of a linear map
$L_m=(B_m,C_m): \R_u^{p} \longrightarrow T_{P_0^m}N\times\R_s^{p}$,
whose norm satisfies:
$$\lim\limits_{m\rightarrow
\infty}\|L_m\|=0. $$
\end{prop}

\begin{proof}


We start with a quick study of the dynamics in $\widetilde{W}^s(N)$.
Recall that $P=(x,s,0)$ is the intersection point of $\widetilde{\d}$ and $\widetilde{W}^s(N)$. Note first that by Proposition~\ref{vois},
$\|\partial_sF_s(P^i)\|<\overline{\lambda}$, for all
$i\geq0$. For $i\geq1$, we set $s_i:=F_s(P^{i-1})$. Then by the Mean
Value Theorem, one gets $\|s_i\|\leq \overline{\lambda}\|s_{i-1}\|$,
and thus under iteration $\|s_i\|\leq \overline{\lambda}^i\|s\|$,
that tends to $0$ with an exponential speed.\\

\noindent $\bullet$ We will see now where the graph property appears. By
transversality of $\widetilde{\d}$ and $\widetilde{W}^s(N)$, and since $\dim \widetilde{\d}=p$,
$T_P\widetilde{\d}$ is the graph of a linear map defined on
$\R_u^{p}$, with values in $T_{P_0}N\times\R_s^{p}$. More precisely,
recall that $\Pi_3 : N \times B_s^p\times B_u^p \longrightarrow
B_u^p$ is the projection over the third variable. By transversality,
$D{\Pi_3}_{\mid_{\widetilde{\d}}}(P)$ is an isomorphism between
$T_P\widetilde{\d}$ and $\R_u^{p}$. Therefore, there exist two
linear maps $B$ and $C$ on $\R_u^{p}$, such that $T_P\widetilde{\d}$
is the image of the map $$(B,C,I):\R_u^{p}\longrightarrow
T_{P_0}N\times\R_s^{p}\times\R_u^{p},$$
where $I: \R_u^{p} \rightarrow \R_u^{p}$ is the identity map.\\

\noindent $\bullet$ Let us now see how the property of $T_P\widetilde{\d}$
being a graph of a linear map persists under iteration. We will
proceed by induction. However, since the calculations are similar
for all the iterates, we will content ourselves with detailing the
proof for the
first iteration.\\

 The image of $T_P\widetilde{\d}$ under $DF(P)$
is $T_{F(P)}F(\widetilde{\d})$. For notational convenience, we will
identify our linear maps with the matrices below (in the suitable
algebras of linear applications) and the partial derivatives with the blocks in the matrices. For instance,
$T_{F(P)}F(\widetilde{\d})$ is identified with the image of the
linear map
$$DF(P) . \left(\begin{array}{c} B \\ C\\ I \end{array}\right): \R_u^{p}\longrightarrow T_{P^1_0}N\times\R_s^{p}\times\R_u^{p}.$$
Since $P$ lies in $\widetilde{W}^s(N)$, this is nothing but the
image of the following map

$$ \left(\begin{array}{ccc}
\partial_xF_x(P)& 0 &  \partial_uF_x(P)\\
\partial_xF_s(P) & \partial_sF_s(P) & \partial_uF_s(P)\\
0& 0 & \partial_uF_u(P)
\end{array}\right ) \left(\begin{array}{c} B \\ C\\ I \end{array}\right)=\left(\begin{array}{c}
\partial_xF_x(P).B + \partial_uF_x(P)\\
\partial_xF_s(P).B + \partial_sF_s(P).C + \partial_uF_s(P)\\
\partial_uF_u(P)
\end{array}\right ).$$

\noindent Since $\partial_uF_u(P) : \R_u^{p}\longrightarrow
\R_u^{p}$ is invertible, $T_{F(P)}F(\widetilde{\d})$, that is, $T_{P^1}\widetilde{\d}^1$ coincides with the
image of

$$\left(\begin{array}{c}
\partial_xF_x(P).B + \partial_uF_x(P)\\
\partial_xF_s(P).B + \partial_sF_s(P).C + \partial_uF_s(P)\\
\partial_uF_u(P)
\end{array}\right ) . (\partial_uF_u(P))^{-1}=$$

$$\left(\begin{array}{c}
\partial_xF_x(P).B.(\partial_uF_u(P))^{-1} + \partial_uF_x(P).(\partial_uF_u(P))^{-1}\\
\partial_xF_s(P).B.(\partial_uF_u(P))^{-1} + \partial_sF_s(P).C.(\partial_uF_u(P))^{-1} + \partial_uF_s(P).(\partial_uF_u(P))^{-1}\\
I
\end{array}\right ).$$

\noindent This shows that $T_{P^1}\widetilde{\d}^1$ is also a graph.
It is the image of the linear map $$(B_1,C_1,I):
\R_u^{p}\longrightarrow
T_{P^1_0}N\times\R_s^{p}\times\R_u^{p},$$ where we have set\\
$$B_1=\partial_xF_x(P).B.(\partial_uF_u(P))^{-1} +
\partial_uF_x(P).(\partial_uF_u(P))^{-1},$$ and
$$C_1=\partial_xF_s(P).B.(\partial_uF_u(P))^{-1} +
\partial_sF_s(P).C.(\partial_uF_u(P))^{-1} +
\partial_uF_s(P).(\partial_uF_u(P))^{-1}.$$


\noindent Pursuing the induction, one gets $B_i$ and $C_i$ ($i>1$),
by applying $DF(P^{i-1})$ to $T_{P^{i-1}}\widetilde{\d}^{i-1}$ (which is the image of $(B_{i-1},C_{i-1},I)$), and
then normalizing by $(\partial_uF_u(P^{i-1}))^{-1}$. We set
$b_i=\|B_i\|$ and $c_i=\|C_i\|$, for $i$ $\in$ $\N$, where $B_0=B$
and $C_0=C$.

\bigskip

\noindent $\bullet$ To end the proof, it is enough now to prove that
$(b_i)$ and $(c_i)$ converge to $0$. We begin with $(b_i)$. We fix
an arbitrary $\varepsilon>0$. Proposition~\ref{vois} yields, for all
$i$ $\in$ $\N$,
$$\|\partial_xF_x(P^{i})\|.
\|(\partial_uF_u(P^i))^{-1}\|<\overline{\lambda},$$so that, since
$\|(\partial_uF_u(P^{i}))^{-1}\|<1$,
$$\begin{array}{lll}b_{i+1}&\leq &
\|\partial_xF_x(P^{i})\|.b_{i}.\|(\partial_uF_u(P^{i}))^{-1}\| +
\|\partial_uF_x(P^{i}))\|.\|(\partial_uF_u(P^i))^{-1}\|\leq\overline{\lambda}
b_{i} + \beta_{i},\end{array}$$ where we have set
$\beta_i:=\|\partial_uF_x(P^{i}))\|$. Therefore, for $n$ $\in$
$\N^*$,

$$b_n\leq \overline{\lambda}^n b_0 + \sum_{i=0}^{n-1}\overline{\lambda}^i\beta_{n-1-i}.$$

\noindent Note that we are not interested in giving the optimal expression for the convergence. Since $\overline{\lambda}<1$, then for $n$ large enough,
$\overline{\lambda}^n b_0 \leq \frac{\varepsilon}2$. On the other
hand, by the Mean Value Theorem, $\beta_i$ satisfies:
$$\beta_i \leq C_2\overline{\lambda}^i\|s\|$$since $\|\partial_uF_x(P_0^{i}))\|=0$. As a consequence of Proposition \ref{vois}, it is easy to
see that $C_2\|s\| \leq 1$. Therefore $$
\sum_{i=0}^{n-1}\overline{\lambda}^i\beta_{n-1-i}  \leq
\sum_{i=0}^{n-1} \overline{\lambda}^i \overline{\lambda}^{n-1-i}
  \leq  \sum_{i=0}^{n-1}\overline{\lambda}^{n-1}= n. \overline{\lambda}^{n-1}.$$Since $\overline{\lambda}<1$, then for $n$ large enough,
$n. \overline{\lambda}^{n-1} \leq \frac{\varepsilon}2$. Then, for
$n$ large enough, $b_n\leq\varepsilon$.\\

Note that one can also prove that the series $\sum b_i$ is convergent. This will be needed for the convergence of
$(c_i)$.\\

Let us now study the convergence of the sequence $(c_i)$. For
$i\geq0$,
$$\begin{array}{ll}c_{i+1}&\leq \|\partial_xF_s(P^i)\|.b_i.\|(\partial_uF_u(P^i))^{-1}\|
+\|\partial_sF_s(P^i)\|.c_i.\|(\partial_uF_u(P^i))^{-1}\|\\ & +
\|\partial_uF_s(P^i)\|.\|(\partial_uF_u(P^i))^{-1}\|.\end{array}$$It
is easy to see, using the Mean Value Theorem, that
$\|\partial_xF_s(P^i)\|<C_2\|s_i\|<1$, for all $i$. As we did for
$(b_i)$, we define $\gamma_i:=\|\partial_uF_s(P^i)\|$ and get
$\gamma_i\leq\overline{\lambda}^i$, following the same steps as for
$\beta_i$. Therefore, $$c_{i+1}\leq b_i + \overline{\lambda}c_i +
\overline{\lambda}^i,$$ and, for $n\geq1$,
$$c_n\leq \sum_{i=0}^{n-1}\overline{\lambda}^{(n-1-i)}b_i + \overline{\lambda}^n c_0 +
\sum_{i=0}^{n-1}\overline{\lambda}^{(n-1-i)}\overline{\lambda}^i.
$$Since $\overline{\lambda}<1$, for $n$ large enough, one gets
$\overline{\lambda}^nc_0\leq\frac{\varepsilon}3$. On the other hand,
for $n$ large enough,
$\sum_{i=0}^{n-1}\overline{\lambda}^{(n-1-i)}.\overline{\lambda}^{i}=n
\overline{\lambda}^{n-1}\leq\frac{\varepsilon}3$.
Finally, let
$s_{n-1}:=\sum_{i=0}^{n-1}\overline{\lambda}^{(n-1-i)}b_i$. Observe
that $s_n$ is the general term of the Cauchy product of the series
of general terms $b_i$ and $\overline{\lambda}^i$ respectively.
These series are both convergent, so is their Cauchy
product. Then $(s_n)$ converges to $0$. More precisely, for $n$
large enough, one has $s_{n-1}\leq \frac{\varepsilon}3$. This ends
the proof of Proposition~\ref{etape1}.
\end{proof}


\subsection{The graph property for $\widetilde{\d}$}\label{step2}

We have seen above that, because of the transversality, $D\Pi_3(P)$
restricted to $T_P\widetilde{\d}$ is an isomorphism between
$T_P\widetilde{\d}$ and $\R_u^{p}$. Then, by the Inverse Function
Theorem, there exist a neighborhood $\mathcal{O}_1$ of $P$ in
$\widetilde{\d}$ and a neighborhood $\mathcal{O}_2$ of $0$ in
$\R_u^{p}$ such that ${\Pi_3}_{\mid_{\widetilde{\d}}}$ is a
diffeomorphism from $\mathcal{O}_1$ onto $\mathcal{O}_2$. More
precisely, there exists a real number $\widetilde{\delta}>0$ such
that, if we set $B_{\widetilde{\delta}}:=\{u\in B_u^p; \|
u\|<\widetilde{\delta}\}$ and $D_{\widetilde{{}\delta}}:=\{(x,s,u)
\in \mathcal{V}; u \in B_{\widetilde{\delta}}\}$, then there exists
a $C^1$-map $\xi : B_{\widetilde{\delta}}\rightarrow N\times B_s^p$,
such that $\widetilde{\d}\cap D_{\widetilde{\delta}}$ is the graph
of $\xi$ (in the sense of Definition \ref{graph}). We set $\xi=(X,
S)$.


\subsection{The graph property for the iterates $\widetilde{\d}^n$ over a fixed strip}\label{step3}

We set $\widetilde{\nu}:=\|\xi'\|=\max(\|X'\|,\|S'\|)=\sup_{u \in
B_{\delta}}(\|\xi'(u)\|)$ and $\nu:=\max(1,\widetilde{\nu})$. We
will see later on why we choose $\nu$ (and not
just $\widetilde{\nu}$) to bound the norm of all the derivatives of
the graph maps. Let us set





\begin{equation}\label{epsilunnu}\varepsilon_{\nu}=\frac{1-\lambda}{12\nu(1+\lambda)}=\frac{1-\overline{\lambda}}{12\nu\overline{\lambda}}.\end{equation}
The reason behind this choice will be clarified later on. By uniform
continuity, and due to the form of $DF$ on $\widetilde{W}^s(N)$ (Equation~(\ref{DFS})),
there exists $\eta>0$, such that for all $(x,s,u)$ $\in$
$\mathcal{V}$, if $\|u\|<\eta$, then

\begin{equation}\label{{epsilunnu'}} \|\partial_xF_u(x,s,u)\|<\varepsilon_{\nu}\,\, \text{ and
}\,\, \|\partial_sF_u(x,s,u)\|<\varepsilon_{\nu}.\end{equation}

\noindent We then set
\begin{equation}\label{deltaa}\delta:=\min\left(1,\widetilde{\delta},\eta,\frac{1-\overline{\lambda}}{3C_2(2\nu+1)^2}\right).\end{equation}

\begin{prop}\label{delta}
Let $\delta$ and $\nu$ be as above. Then, for all $n$\, $\in$ $\N$,
there exists a $C^1$-map $\xi_n:B_{\delta}\rightarrow N\times B_s^p$
such that $\widehat{\Delta}^n:=\widetilde{\d}^n\cap D_{\delta}$ is
the graph of $\xi_n$. Moreover, if for all $n$ $\in$ $\N$,
$\xi_n=(X_n,S_n)$, then $\|\xi'_n\|:=\max(\|X'_n\|,\|S'_n\|)=\sup_{u
\in B_{\delta}}(\|\xi'_n(u)\|)$ satisfies $\|\xi'_n\|\leq\nu$.
\end{prop}

 The rest of this subsection is devoted to the proof of
Proposition \ref{delta}. We will proceed by induction. We first
prove these statements for the first iteration, by using
intermediate lemmas which will be very useful for the estimates
later on. All the computations will be independent of $n$, which will
easily yield the proof of the inductive step.



Note that when $n=0$, the statement follows from Section \ref{step2}
and the definition of $\nu$. Therefore, we have to prove that if for
$n$ $\in$ $\N$, $\widehat{\Delta}^n=\graph\xi_n=\{(X_n(u),S_n(u),u);
u \in B_{\delta}\}$ with $\|\xi'_n\|\leq\nu$, then
$F(\widehat{\Delta}^n)$ is also the graph of a map $\xi_{n+1}$ over an
open set in $\R_u^{p}$ strictly containing $B_{\delta}$. We then set
$$\widehat{\Delta}^{n+1}=F(\widehat{\Delta}^n)\cap
D_{\delta}=\widetilde{\d}^{n+1}\cap
D_{\delta}=\graph\xi_{n+1}=\{(X_{n+1}(u),S_{n+1}(u),u); u \in
B_{\delta}\}.$$ Note that we will keep the same notation for $\xi_{n+1}$ and its restriction to $B_{\delta}$. We also have to prove that $\|\xi'_{n+1}\|<\nu$.

To simplify this step and to keep our formulas legible, we will
actually prove that Proposition~\ref{delta} holds true when $n=1$.
Since all the computations will be independent of $n$, one can
easily see that the statements are valid for an arbitrary $n$.

By applying $F$ to $\widehat{\Delta}=\{(X(u),S(u),u); u \in
B_{\delta}\}$, one gets
$$F(\widehat{\Delta})=\big\{\big(F_x(X(u),S(u),u),F_s(X(u),S(u),u),F_u(X(u),S(u),u)\big);
u \in B_{\delta}\big\}.$$ Let $G(u):=F_u(X(u),S(u),u)=h$. We will prove
that $G$ is a homeomorphism onto its image $B'_{\delta}$ and that
the latter strictly contains $B_{\delta}$. Then, it is easy to see
that $F(\widehat{\Delta})$ restricted to $D'_{\delta}:=\{(x,s,u) \in
\mathcal{V}; u \in B'_{\delta}\}$ is the graph of $(X_1,S_1)$, where
$$X_1(h)=F_x\big(X(G^{-1}(h)), S(G^{-1}(h)),G^{-1}(h)\big), $$ and
$$S_1(h)=F_s\big(X(G^{-1}(h)), S(G^{-1}(h)),G^{-1}(h)\big), $$
for $h \in B'_{\delta}$. We will need the following lemmas.

\begin{lemme}\label{G'u} For all $u$ $\in$ $B_{\delta}$, $G'(u)$ is an
isomorphism on $\R_u^{p}$. Moreover,
$$[G'(u)]^{-1}=\Big(\sum_{m\geq0}\left (-H(u)\right )^m\Big).[\partial_uF_u(X(u),S(u),u)]^{-1},$$
where
$H(u):=[\partial_uF_u(X(u),S(u),u)]^{-1}.[\partial_xF_u(X(u),S(u),u).X'(u)
+
\partial_sF_u(X(u),S(u),u).S'(u)]$.
\end{lemme}

\begin{proof}
$G(u)=F_u(X(u),S(u),u)$ gives by derivation
$$G'(u)=\partial_xF_u(X(u),S(u),u).X'(u)
+\partial_sF_u(X(u),S(u),u).S'(u) +
\partial_uF_u(X(u),S(u),u).$$
Recall that the linear map $\partial_uF_u(x,s,u)$ is invertible for
all $(x,s,u)$ $\in$ $\mathcal{V}$, and satisfies:
$$\|[\partial_uF_u(x,s,u)]^{-1}\|<\overline{\lambda}<1.$$ Then one
can write
$$G'(u)=[\partial_uF_u(X(u),S(u),u)].[I + H(u)],$$ where
$H(u):=[\partial_uF_u(X(u),S(u),u)]^{-1}.[\partial_xF_u(X(u),S(u),u).X'(u)
+
\partial_sF_u(X(u),S(u),u).S'(u)]$. Since
$\partial_uF_u(X(u),S(u),u)$ is invertible, it is enough to prove
that $I + H(u)$ is invertible too. It is the case if $\|H(u)\|<1$ because it
is an endomorphism of $\R_u^{p}$. So now we will prove that
$\|H(u)\|<1$. It is easy to see that, by definition of
$\varepsilon_{\nu}$ (equation~(\ref{epsilunnu})), for all $u$ $\in$
$B_{\delta}$, $$\|H(u)\| < 2\overline{\lambda} \nu
\varepsilon_{\nu}=\frac{1-\overline{\lambda}}6<1.$$

\noindent Therefore $I+H(u)$ is invertible on $\R_u^{p}$ and $[I +
H(u)]^{-1}=\sum_{m\geq0}(-H(u))^m$. This ends the proof of Lemma
\ref{G'u}.
\end{proof}

\begin{lemme}\label{A} For all $u$ $\in$ $B_{\delta}$,
$\|[G'(u)]^{-1}\|<1$.
\end{lemme}

\begin{proof} This easily follows from the previous lemma. In
fact,
$$\begin{array}{cll}
 \|[G'(u)]^{-1}\|&\leq&\|\sum\limits_{m\geq0}\left (-H(u)\right )^m\|.
 \|[\partial_uF_u(X(u),S(u),u)]^{-1}\|\\&&\\&\leq& \frac1{1-\|H(u)\|}.
 \|[\partial_uF_u(X(u),S(u),u)]^{-1}\|\\&&\\
&<& \frac{1}{1-2\overline{\lambda}\nu\varepsilon_{\nu}}.
  \|[\partial_uF_u(X(u),S(u),u)]^{-1}\|\\&&\\

 &<& \frac6{5+\overline{\lambda}}.
  \|[\partial_uF_u(X(u),S(u),u)]^{-1}\|<\frac{6\overline{\lambda}}{5+\overline{\lambda}}<1.
\end{array}$$
\end{proof}

We will now prove that $G$ is invertible.
\begin{prop}\label{prop1} There exists an open set $B'_{\delta}$ in
$\R_u^p$ strictly containing $B_{\delta}$, such that $G$ is a
homeomorphism from $B_{\delta}$ onto $B'_{\delta}$.

\end{prop}
\begin{proof}
Without loss of generality, we can assume that $\xi$ is defined on
$\overline{B}_{\delta}$. We introduce an auxiliary map defined on
$\overline{B}_{\delta}$,
$$\chi(u)=[\partial_uF_u(X(0),S(0),0)]^{-1}. G(u)=[\partial_uF_u(P)]^{-1}. G(u).$$

\noindent We will first study the invertibility of $\chi$, from
which that of $G$ easily follows. Let $y$ be in a subset of $\R_u^p$ to
be specified later on. We are looking for the conditions under which
there exists a unique $x$ $\in$ $B_{\delta}$, such that $y=\chi(x)$.
We let $\psi(x):=x-\chi(x)+y$, so that the point $y$ has a unique preimage
under $\chi$ if and only if $\psi$ has a unique fixed point. To
prove this last property, we will need the next lemma.

\begin{lemme}\label{contractante} For all $u$ $\in$ $\overline{B}_{\delta}$, $\|I-\chi'(u)\|<\frac{1-\overline{\lambda}}2$.

\end{lemme}

\begin{proof} By derivating $\chi(u)=[\partial_uF_u(P)]^{-1}.
G(u)$, one gets
$$\chi'(u)=[\partial_uF_u(P)]^{-1}.[\partial_uF_u(X(u),S(u),u)].[I +
H(u)].$$ We set
$\mathcal{W}:=[\partial_uF_u(P)]^{-1}.[\partial_uF_u(X(u),S(u),u)]$
and $\mathcal{T}:=\mathcal{W}-I$ so that $\mathcal{W}=\mathcal{T}+I$
and $\chi'(u)-I=\mathcal{W}.(I+ H(u))-I$. Recall that
$\|H(u)\|<2\overline{\lambda} \nu
\varepsilon_{\nu}=\frac{1-\overline{\lambda}}6$ (see the proof of
Lemma~\ref{G'u}). Therefore,
$$\begin{array}{ll}
\|\mathcal{T}\|&=\|[\partial_uF_u(P)]^{-1}.[\partial_uF_u(X(u),S(u),u)]-I\|\\
&=\|[\partial_uF_u(P)]^{-1}.[\partial_uF_u(X(u),S(u),u)-\partial_uF_u(P)]\|\\
&\leq\|[\partial_uF_u(P)]^{-1}\|.\|[\partial_uF_u(X(u),S(u),u)-\partial_uF_u(X(0),S(0),0)\|\\
&\leq\overline{\lambda}(2C_2\nu+C_2)\|u\|,\\
\end{array}$$
by the Mean Value Theorem. Writing $\chi'(u)-I=(\mathcal{T}+I).(I+
H(u))-I=H(u) + \mathcal{T}.(I+H(u))$  gives
$$\begin{array}{ll}
\|\chi'(u)-I\|&\leq\|H(u)\|+\|\mathcal{T}\|.(1+\|H(u)\|)\\
&<  2\overline{\lambda} \nu
\varepsilon_{\nu}+ \overline{\lambda}(2C_2\nu+C_2)\|u\| (1+2 \varepsilon_{\nu} \overline{\lambda}\nu) \\
&< \frac{1-\overline{\lambda}}6 + C_2(2\nu+1)^2\|u\|,
\end{array}$$because $\overline{\lambda}<1$ and $\overline{\lambda}
\varepsilon_{\nu}<1$ using equation (\ref{epsilunnu}).

Recall that $\|u\|< \frac{1-\overline{\lambda}}{3C_2(2\nu+1)^2}$ by
equation (\ref{deltaa}), which yields
$$\begin{array}{ll}
\|\chi'(u)-I\|& <\frac{1-\overline{\lambda}}6 +
\frac{1-\overline{\lambda}}3\\
&<\frac{1-\overline{\lambda}}2.\\
\end{array}$$

\noindent This ends
the proof of Lemma \ref{contractante}.
\end{proof}

\noindent $\bullet$ We now go back to proving the invertibility of
$\chi$. Let $\kappa:=\frac{1-\overline{\lambda}}2$. Clearly $\kappa<1$. The last lemma
shows that $\psi=I_{B_{\delta}}-\chi+y$ is a contracting map. In
order for it to have a unique fixed point, one needs to have
$\psi(\overline{B}_{\delta})\subset\overline{B}_{\delta}$. And this
condition is satisfied if $\|y\|\leq\delta(1-\kappa)$. Therefore
$\chi:\overline{B}_{\delta} \longrightarrow \Im\chi$ is bijective
and satisfies $\overline{B}_{\delta(1-\kappa)}\subset\Im\chi$.\\


\noindent $\bullet$ The invertibility of $G$ easily follows from
that of $\chi$. Recall that $\chi=[\partial_uF_u(P)]^{-1}.G$.
Therefore, $G:\overline{B}_{\delta} \longrightarrow \Im G$ is an
homeomorphism and satisfies $B'_{\delta}:= \Im
G\supset\overline{B}_{\frac{\delta(1-\kappa)}{\overline{\lambda}}}$.

\noindent Recall that $\kappa=\frac{1-\overline{\lambda}}2$ which
gives $(1-\kappa)>\overline{\lambda}$ and thus $B'_{\delta}$, which
contains $B_{\frac{\delta(1-\kappa)}{\overline{\lambda}}}$, strictly
contains $B_{\delta}$. This ends the proof of the proposition.
\end{proof}

\bigskip

Therefore, the proof of the graph property in Proposition
\ref{delta} for the case $n=1$ is complete. Let
$\widehat{\Delta}^{1}=F(\widehat{\Delta})\cap
D_{\delta}=\graph\xi_{1}=\{(X_{1}(u),S_{1}(u),u); u \in
B_{\delta}\}$. The next proposition will not only end the proof of
the case $n=1$, but will also be a preliminary step to estimating
$\lim\limits_{n\rightarrow \infty} \|\xi'_n\|$ in Section
\ref{step4}.


\begin{prop}\label{xi'} If we set $\|\xi'_1\|:=\sup_{u \in B_{\delta}}(\|\xi_1'(u)\|)
=\max(\|X'_1\|,\|S'_1\|)$, then $\|\xi'_1\| <\nu$.
\end{prop}

\begin{proof} We recall that for $h$ $\in$ $B'_\delta$,
$$X_1(h)=F_x\left(X(G^{-1}(h)), S(G^{-1}(h)),G^{-1}(h)\right), $$ and
$$S_1(h)=F_s\left(X(G^{-1}(h)), S(G^{-1}(h)),G^{-1}(h)\right). $$

\noindent Since we are only interested in uniform norms over
$B_\delta$, we consider $h$ to belong to $B_\delta$ from now on.
We let $u:=G^{-1}(h)$. Then $u$ $\in$ $G^{-1}(B_\delta) \varsubsetneq B_\delta$. We write
$$X_1(G(u))=F_x(X(u), S(u),u), $$ and
$$S_1(G(u))=F_s(X(u), S(u),u). $$
By derivating the two sides with respect to $u$ and inverting
$G'(u)$, one gets for all $u$ $\in$ $B_\delta$
$$\begin{array}{ll}
X'_1(G(u))&=\partial_xF_x(X(u),S(u),u).X'(u).[G'(u)]^{-1}+\partial_sF_x(X(u),S(u),u).S'(u).[G'(u)]^{-1}\\&
+\,\,\partial_uF_x(X(u),S(u),u).[G'(u)]^{-1} ,\end{array}$$ and
$$\begin{array}{ll}
S'_1(G(u))&=\partial_xF_s(X(u),S(u),u).X'(u).[G'(u)]^{-1}+\partial_sF_s(X(u),S(u),u).S'(u).[G'(u)]^{-1}
\\&+\,\,\partial_uF_s(X(u),S(u),u).[G'(u)]^{-1}.\end{array}$$

\noindent Let us begin by studying
$T:=\|\partial_xF_x(X(u),S(u),u)\|.\|[G'(u)]^{-1}\|$. Using the
estimates in Lemma~\ref{A}, one gets
$$\begin{array}{ll}T&\leq
\|\partial_xF_x(X(u),S(u),u)\|.\|[I+H(u)]^{-1}\|.\|[\partial_uF_u(X(u),S(u),u)]^{-1}\|\\&\\&<
\overline{\lambda}.\|[I+H(u)]^{-1}\|<\frac{6\overline{\lambda}}{5+\overline{\lambda}}:=\widetilde{\alpha},\end{array}$$
where we can easily see that
$0<\overline{\lambda}<\widetilde{\alpha}<1$. Recall that by
Proposition \ref{vois}, for all $u$ $\in$ $B_\delta$,
$$\max\big(\|\partial_sF_x(X(u),S(u),u)\|,\|\partial_xF_s(X(u),S(u),u)\|\big)<\frac{5-5\lambda}{2(11+\lambda)},$$which
yields
$$\|\xi'_{1}\|<\big(\widetilde{\alpha}+\frac{5-5\lambda}{2(11+\lambda)}\big)\|\xi'\|+\sup_{u \in B_{\delta}}\max\big(\|\partial_uF_x(X(u),S(u),u)\|,
\|\partial_uF_s(X(u),S(u),u)\|\big).$$

\noindent On the one hand,
$\widetilde{\alpha}=\frac{6\overline{\lambda}}{5+\overline{\lambda}}=\frac{6+6\lambda}{11+\lambda}$,
and thus
$\widetilde{\alpha}+\frac{5-5\lambda}{2(11+\lambda)}=\frac{1+\widetilde{\alpha}}2:=\beta$,
with $0<\widetilde{\alpha}<\beta<1$. On the other hand, using the
particular form of $F$ on the unstable manifold
(Equations (\ref{Fs}) and (\ref{Fx})), for $X$ $\in$ $\widetilde{W}^u(N)$, the
derivative $DF(X)$ at $X$ has the following form
\begin{equation}\label{DFU}DF(X)=\left(\begin{array}{ccc}
\partial_xF_x(X)& \partial_sF_x(X) &  0\\
0 & \partial_sF_s(X) & 0\\
\partial_xF_u(X)& \partial_sF_u(X) & \partial_uF_u(X)
\end{array}\right ).\end{equation} Therefore, using this particular form and the Mean Value Theorem,
one gets

$$\sup_{u \in B_{\delta}}\max\big(\|\partial_uF_x(X(u),S(u),u)\|,\|\partial_uF_s(X(u),S(u),u)\|\big)\leq
C_2 \sup_{u \in B_{\delta}}\|S(u)\|,$$and so \begin{equation}\label{xi'_1}\|\xi'_{1}\|<\beta
\|\xi'\|+C_2 \sup_{u \in B_{\delta}}\|S(u)\|.\end{equation}

\noindent Using item 1 of Proposition \ref{vois}, and the fact that
$\nu\geq1$, one gets

$$\|\xi'_{1}\|<\left(\beta +\frac{5-5\lambda}{2(11+\lambda)}\right)\nu=\nu.$$

\noindent This ends the proof of the lemma.
\end{proof}


Observe that the fact that $\nu$ is larger than $1$ is crucial
to show that $\|\xi'_{1}\|<\nu$ which explains our initial choice in the beginning of Section~\ref{step3}.  Since all the computations in the previous lemmas are
independent of $n$, the proof of the inductive step easily follows.

We then set $\widehat{\Delta}^{n}=\widetilde{\d}^{n}\cap
D_{\delta}=\graph\xi_{n}=\{(X_{n}(u),S_{n}(u),u); u \in
B_{\delta}\}$  for all $n$ $\in$ $\N$. This ends the proof of
Proposition \ref{delta}.




\subsection{Linear straightening along the graphs}\label{step4}

We will now see how tangent vectors along the graphs are
straightened. We will use the estimates of the previous section
to prove the following proposition.

Recall that
$\widehat{\Delta}^{n}=\graph\xi_{n}=\graph(X_{n},S_n)=\{(X_{n}(u),S_{n}(u),u);
u \in B_{\delta}\}$  for all $n$ $\in$ $\N$.

\begin{prop}
For all $\varepsilon>0$, there exists $n_0$ $\in$ $\N$, such that
for all $n\geq n_0$, $\|\xi'_n\|<\varepsilon$.
\end{prop}

\begin{proof}  Generalizing to all the iterates Inequality~(\ref{xi'_1}), since the estimates are uniform with
respect to the order of the iteration, one gets
\begin{equation}\label{xi'_n+1}\|\xi'_{n+1}\|<\beta\|\xi'_{n}\|+C_2\sup_{u \in B_{\delta}}\|S_n(u)\|.\end{equation}By the Mean Value Theorem, one can prove by induction that $\sup\limits_{u \in
B_{\delta}}\|S_n(u)\|\leq \overline{\lambda}^n\sup\limits_{u \in
B_{\delta}}\|S(u)\|$. More precisely, for all $u$ $\in$ $B_{\delta}$
and for all $n$ $\in$ $\N^*$, there exists $Z=(Z_1,Z_2,Z_3)$ $\in$
$\widehat{\d}^{n-1}$ such that
$S_n(u)=F_s(Z)=F_s(Z)-F_s(Z_1,0,Z_3)$. Therefore $\|S_n(u)\|\leq
\overline{\lambda}\|Z_2\|\leq \overline{\lambda}\sup_{u \in
B_{\delta}}\|S_{n-1}(u)\|$, since
$\widehat{\d}^{n-1}=\graph(X_{n-1},S_{n-1})$, which proves our claim. Since
$C_2\sup\limits_{u \in B_{\delta}}\|S(u)\|<1$, then by Inequality~(\ref{xi'_n+1}),

$$\|\xi'_{n+1}\|<\beta\|\xi'_{n}\|+\overline{\lambda}^n.$$



\noindent The proof of the convergence follows the same lines as
that of $(b_n)$ in Section \ref{step1}, since $\beta<1$.
\end{proof}

\subsection{Nonlinear straightening and proof of Theorem 1}\label{C--0}

We can now end the proof of Theorem 1 by a simple application of the
Mean Value Theorem. We get for $n\geq n_0$,
$$\begin{array}{ll}\sup\limits_{u \in B_{\delta}}d\big(\xi_n(u),\left(P_0^n,0\right)\big)&\leq\sup\limits_{u \in
B_{\delta}}d\big(\xi_n(u),\xi_n(0)\big)+d\big(\xi_n(0),(P_0^n,0)\big)\\
&<\varepsilon+\|S_n(0)\|\\
&<\varepsilon+\overline{\lambda}^n\|S(0)\|,\end{array}$$ where we
have used that $\|u\|<1$. The convergence easily follows. This
completes the proof of Theorem 1.

\section{Proof of Theorem 2}\label{inK}

\begin{figure}
\centering
\def\JPicScale{0.8}
\input{proofcompactK.pst}
\caption{In $f^m(\overline{U})$}
\end{figure}
We will now prove Theorem 2 which will be a consequence of Theorem~1. Let $\varphi$ be the diffeomorphism
given by Proposition B and $\mathcal{U}$ be as in Section \ref{not}.
Let $\delta$ be given by Theorem 1. We set
$\overline{U}:=\varphi^{-1}(D_{\delta})$. Let $m$ $\in$ $\N$ be
fixed, then
$\psi_{(m,\overline{U})}=f^m\circ{\varphi^{-1}}_{|_{D_{\delta}}}$.
Let $(\widehat{\d}^n)$ be as in Theorem 1. For all $k\geq m$, let
$\overline{\d}^k:=\psi_{(m,\overline{U})}(\widehat{\d}^{k-m})$.
The $(m,\overline{U})$-graph property of $\overline{\d}^k$ is an
immediate consequence of Theorem~1. As for the convergence, the
$C^0$ part of the convergence is obvious by uniform continuity of
$\psi_{(m,\overline{U})}$. It is now enough to prove the
convergence of the second term of the
$(C^1,m,\overline{U})$-distance. There exist two positive real
numbers $\overline{C}$ and $\overline{\overline{C}}$, such that for
all $u$ $\in$ $B_{\delta}$, for all $v_1$ $\in$ $B_{\R_u^{p}}$, for
all $n$ $\in$ $\N$, if we set
$T:=\|{D\psi_{(m,\overline{U})}}(\xi_n(u),u).({\xi_n}'(u).v_1,v_1)-{D\psi_{(m,\overline{U})}}(X_n(0),0,u).(0,v_1)\|$,
then
$$\begin{array}{rl}
T\,\leq&
\|{D\psi_{(m,\overline{U})}}(\xi_n(u),u).({\xi_n}'(u).v_1,v_1)-{D\psi_{(m,\overline{U})}}(\xi_n(u),u).(0,v_1)\|
\\&+
\|{D\psi_{(m,\overline{U})}}(\xi_n(u),u).(0,v_1)-{D\psi_{(m,\overline{U})}}(X_n(0),0,u).(0,v_1)\|\\
\leq& \|{D\psi_{(m,\overline{U})}}(\xi_n(u),u)\|.\|{\xi_n}'(u)\|+
\|{D\psi_{(m,\overline{U})}}(\xi_n(u),u)-{D\psi_{(m,\overline{U})}}(X_n(0),0,u)\|\\
 \leq&
\overline{C}.\|{\xi_n}'(u)\|+\overline{\overline{C}}
d\left((\xi_n(u),u),(X_n(0),0,u)\right)
\end{array}$$by the Mean Value Theorem. The convergence follows from
Theorem 1. By setting $n:=k-m$, the proof of Theorem 2 is now
complete.



\section{Application to diffusion}\label{app1}
We will now use the $\lambda$-lemma to prove a diffusion
result. We will prove the existence of a shadowing orbit for a
finite family of invariant dynamically minimal sets contained in a normally hyperbolic manifold and having successive
heteroclinic connections. We will see that the existence of Arnold's
diffusion orbit easily follows from this application. 

\subsection{Shadowing
orbits for a finite family of invariant minimal sets} In this
section, we prove a corollary of the $\lambda$-lemma that
gives the existence of a shadowing orbit for a transition chain. 
 Let
$f$, $M$ and $N$ be as in Section~\ref{jfk}. If $A$ is an invariant \emph{dynamically
minimal} set contained in $N$, that is, a set in which the orbit of each point is dense, we set $$W^u(A):=\bigcup\limits_{a\in
A}W^{uu}(a).$$

\begin{definition}\em{\textbf{[Transition chain].}}\label{transitionchain} Let $n \in \N$, $(n>1)$. Let
$(A_k)_{1\leq k\leq n}$ be a finite family of invariant dynamically
minimal sets contained in $N$. We say that $(A_k)$ is a transition
chain if, for all $k=1,\ldots,n-1$, $W^u(A_k)\cap
W^s(A_{k+1})\neq\emptyset$.
\end{definition}

Note that we do not require any regularity for the sets. In the
Hamiltonian nearly integrable case, they can be general Aubry-Mather
sets for instance.\\

 We will only need the convergence in the $C^0$
topology stated in the $\lambda$-lemma to prove the
following result. In Figure \ref{heter}, we  illustrate the
assumptions of Corollary \ref{cor}, in the particular case $n_0=2$
and $p=1$. Of course, since the invariant manifolds are $3$-dimensional, this is only a rough representation of the situation.


\begin{cor}\label{cor}
Let $f$, $M$ and $N$ be as in Section~\ref{jfk}. Let $(A_k)_{1\leq k\leq n}$ be a
transition chain in $N$ such that, for all $k=1,\ldots,n-1$, there
exist $a_k$ $\in$ $A_{k}$, $b_{k+1}$ $\in$ $A_{k+1}$ and $c_k$ $\in$
$W^{uu}(a_{k})\cap W^{ss}(b_{k+1})$ such that $W^{uu}(a_{k})$ and
$W^s(N)$ transversely intersect at $c_k$. Then, for any $\varrho>0$, there exists an orbit $\Gamma$ such that, for all $k=1,\ldots,n$, $\Gamma$ intersects the $\varrho$-neighborhood of $A_k$.


\end{cor}

\begin{figure}
\centering
\def\JPicScale{0.8}
\input{corarticle.pst}
\caption{Heteroclinic connections} \label{heter}
\end{figure}

\begin{proof} We fix $\varrho>0$ and we denote by $\mathcal{V}_{\varrho}(A_k)$ the $\varrho$-neighborhood of $A_k$ for all $k=1,\ldots,n$. Without any loss of generality, we can suppose that, for all $k=1,\ldots,n$,
$\mathcal{V}_{\varrho}(A_k)\subset \mathcal{U}$ (defined in Section~\ref{not}). Therefore, we can
restrict the problem to the straightening neighborhood
$\mathcal{V}\subset V$ (see Proposition~B and Proposition
\ref{vois}).

Fix $k=1,\ldots, n-1$ and fix a ball $B_{k+1}$ centered at some point $z$ of $W^u(A_{k+1})$ and
contained in $\mathcal{V}_\varrho(A_{k+1})$. Since $z$ is in $W^u(A_{k+1})$, there exists a unique
$d_{k+1}$ $\in$ $A_{k+1}$ such that $z$~$\in$~$W^{uu}(d_{k+1})$. Let
us set $\Delta:=W^{uu}(a_{k})$, then $\Delta$ is an immersed
$p$-dimensional $C^3$-submanifold of $M$, transversely intersecting
$W^s(N)$. For $m$ $\in$ $\N^*$, we let $b_{k+1}^m:=F^m(b_{k+1})$ and $\Delta^m$ be the
connected component of $F(\Delta^{m-1})\cap \mathcal{V}$ containing
$c_k^m:=F^m(c_k)$.

$\bullet$ By the $\lambda$-lemma, for all $\varepsilon>0$,
there exists $N_1$ $\in$ $\N$ such that for all $m\geq N_1$,
$\Delta^m$ is $\varepsilon$-close to $W^{uu}(b_{k+1}^m)$ (in the
sense of Definition \ref{epsilunclose}).



$\bullet$ Since $A_{k+1}$ is invariant, the sequence
$(b^m_{k+1})_{m\in \N}$ lies in $A_{k+1}$. Since this set is also
minimal, we can extract a subsequence $(b_{k+1}^{m_j})_{j\in \N}$
such that $\lim\limits_{j\rightarrow \infty}b_{k+1}^{m_j}=d_{k+1}$.
More precisely,

$$\forall \ \varepsilon > 0, \exists \ N_2 \in \N ; \ j\geq N_2\Rightarrow
d(b^{m_j}_{k+1},d_{k+1})<\varepsilon.$$

$\bullet$ The foliations being straightened, for $j$ large enough,
$W^{uu}(b_{k+1}^{m_j})$ is $\varepsilon$-close to $W^{uu}(d_{k+1})$.

$\bullet$ Therefore for $j$ large enough, $\d^{m_j}$ intersects $B_{k+1}$.

Let $y$ be in $\d^{m_j}\cap B_{k+1}$ which is in $W^u(A_k)$. Then, for $q$ large enough, $F^{-q}(y)\in\mathcal{V}_\varrho(A_k)$.
Therefore there exists a ball $B_k$ centered at $F^{-q}(y)$ and contained in $\mathcal{V}_\varrho(A_k)$ 
such that 
$$
F^{q}(B_k)\subset B_{k+1}.
$$
We proved then the existence of a ball $B_k$ centered on $W^u(A_k)$ in $\mathcal{V}_\varrho(A_k)$ and a positive $q$
such that $F^{q}(B_k)\subset B_{k+1}$.

Therefore, given a ball $B_n$ centered on $W^u(A_n)$ and contained in $\mathcal{V}_\varrho(A_n)$, 
an immediate induction proves the existence of an integer $q^*$ and  a ball $B_1$ centered on $W^u(A_1)$ 
and contained in $\mathcal{V}_\varrho(A_1)$, whose sequence of iterates intersects each $\mathcal{V}_\varrho(A_k)$ and
which moreover satisfies $F^{q^*}(B_1)\subset B_n$.
This proves our claim.

%
\end{proof}
%




\nocite{*} \subsection{Particular case: Arnold's example}\label{Arnold1}
We will see in this section that Arnold's system (\cite{Arn64}) satisfies all the assumptions of Corollary~\ref{cor} and thus, one easily
deduces the existence of  drifting orbits. In Arnold's example, the stable manifold of a torus transversely intersects the unstable manifold of the next torus. These manifolds are Lagrangian and the Lagrangian/Lagrangian intersections will easily yield the isotropic/coisotropic intersections needed in Corollary~\ref{cor} ($W^{uu}(a_{k})$ and
$W^s(N)$ transversely intersecting at $c_k$). We start with a reminder on Arnold's example and define the objects ($F$, $M$, $N$, the transition chain,...) needed to set up the context of Corollary~\ref{cor}.

\vspace{1mm}

The {\em autonomous} version of the Hamiltonian used by Arnold is defined
on $\T^3\times\R^3$ and is given by

$$H_{\varepsilon,\mu}(\theta,r)=\frac12(r_1^2+r_2^2)+r_3+\varepsilon
(\cos \theta_1-1)+\varepsilon\mu(\cos \theta_1-1)(\cos \theta_2+\sin\theta_3),$$
where $\theta=(\theta_1,\theta_2,\theta_3)$ $\in$ $\T^3$,
$r=(r_1,r_2,r_3)$ $\in$ $\R^3$ and $0<|\mu|<<|\varepsilon|<<1$.

\begin{theoreme}\em{\textbf{[Arnold]}}
Given $A<B$, there exists $\varepsilon_0>0$ such that for all
$\varepsilon$ $\in$ $]0;\varepsilon_0[$ there exists $\mu_0$ such
that for all $\mu$ $\in$ $]0;\mu_0[$, the system
$H_{\varepsilon,\mu}$ admits an orbit whose projection on the action
space $\R^3$ intersects the open sets $r_2<A$ and $r_2>B$.
\end{theoreme}

The Hamiltonian $H_{\varepsilon,\mu}$ is a perturbation of
$H_{0,0}=\frac12(r_1^2+r_2^2)+r_3$, and the parameters $\varepsilon$
and $\mu$ play asymmetric roles: $\varepsilon$ preserves the
integrability and creates hyperbolicity, and $\mu$ breaks down the
integrability and causes instability. More precisely, when
$\varepsilon=0$, $\T^3\times\R^3$ is foliated by invariant
lagrangian tori, and when $\varepsilon>0$ and $\mu=0$, the system is
equivalent to the uncoupled product of a pendulum
($H_p(\theta_1,r_1)=\frac12r_1^2+\varepsilon(\cos \theta_1-1)$) with
the completely integrable system
$H_r(\theta_2,\theta_3,r_2,r_3)=\frac12r_2^2+r_3$. The resonant
surface given by the equation $r_1=0$, which is invariant and foliated
by invariant tori when $\varepsilon=0$, is destroyed. It gives rise
to a one-parameter family of $2$-dimensional invariant tori which are partially
hyperbolic, whose union is  the normally hyperbolic invariant manifold
$N':=\{0,0\}\times \T^2\times\R^2$. The invariant manifolds of $N'$
are the product of those of the hyperbolic point
$(\theta_1=0,r_1=0)$ with the annulus $\T^2\times\R^2$. When
$|\mu|>0$, we lose the integrability and the invariant manifolds of the tori do
not coincide anymore. The Poincar\'e-Melnikov integrals show that
there exists $\varepsilon_0>0$ such that for all $\varepsilon$ $\in$
$]0;\varepsilon_0[$ there exists $\mu_0$ such that for all $\mu$
$\in$ $]0;\mu_0[$, the invariant manifolds transversely intersect
along a homoclinic orbit. Note that Arnold chose the last term of the perturbation in
such a way that it vanishes on the invariant tori (because
$\theta_1=0$), and thus the previous partially hyperbolic tori are
preserved when $\mu>0$, as well as the normally hyperbolic manifold.

It is possible to choose a section $\mathcal{S}$ (see \cite{Mar96})
contained in an energy level $\mathfrak{H}$ and transverse (in
$\mathfrak{H}$) to the Hamiltonian flow. The Poincar\'e map associated
to $\mathcal{S}$ and defined in a neighborhood of
$N:=N'\cap\mathcal{S}$ (which is also normally hyperbolic) in
$\mathcal{S}$ will play the role of $F$ (this of course is immediate with the 
nonautonomous form of the system). Note that $\mathcal{S}$ can
be chosen so that the invariant manifolds of $N$ are the
intersections of those of $N'$ with $\mathcal{S}$.

Let $\omega$ be irrational and let $T_{\omega}$ be the torus in $N$ given
by the equation $r_2=\omega$. It is invariant and minimal (because
$\omega$ is irrational). Arnold proved the existence of a finite
family $(T_{\omega_i})_{i\in I}$ of those tori that have in addition
Lagrangian invariant manifolds with transverse heteroclinic connections:
$W^u(T_{\omega_i})\pitchfork W^s(T_{\omega_{i+1}})$.

To get Arnold's orbits, it suffices now to apply Corollary \ref{cor} to
the family $(T_{\omega_i})$, since this family is contained in a
normally hyperbolic manifold. The Lagrangian/Lagrangian intersection
implies the isotropic/coisotropic intersection needed in the
corollary. More precisely, for all $i\in I$, let $c_i$ $\in$ $W^u(T_{\omega_i})\pitchfork W^s(T_{\omega_{i+1}})$. We set $a_i$ the point in $T_{\omega_i}$ such that $c_i\in W^{uu}(a_i)$. It is easy to see that $W^{uu}(a_i)$ transversely intersects $W^s(N)$ at $c_i$. One gets then a transition chain as in Corollary~\ref{cor}.\\


\break

\addcontentsline{toc}{section}{References}
\bibliographystyle{amsalpha}
\bibliography{lambda-lemma}
\end{document}